\pdfoutput=1
\documentclass[a4paper,UKenglish,cleveref,thm-restate]{lipics-v2021}

\hideLIPIcs
\nolinenumbers
\makeatletter
\renewenvironment{abstract}{\vskip\topmattervskip\bigskipamount
  \noindent\rlap{\color{lipicsLineGray}\vrule\@width\textwidth\@height1\p@}%
  \hspace*{7mm}\fboxsep1.5mm\colorbox[rgb]{1,1,1}{\raisebox{-0.4ex}{\large\selectfont\sffamily\bfseries\abstractname}}%
  \vskip3\p@
  \fontsize{9}{12}\selectfont
  \noindent\ignorespaces}
{\vskip\topmattervskip\baselineskip\noindent
  \fundingHeading\@funding
  \protected@write\@auxout{}{\string\gdef\string\@pageNumberEndAbstract{\thepage}}}
\let\@oddfoot\@empty
\makeatother

\usepackage{nicefrac}
\usepackage{refcount}

\newtheorem*{gallai}{Gallai's Theorem}
\newtheorem*{halesjewett}{Hales-Jewett Theorem}

\theoremstyle{claimstyle}
\newtheorem{subclaim}{Claim}
\numberwithin{subclaim}{theorem}

\crefname{subclaim}{Claim}{Claims}

\let\leq\leqslant
\let\geq\geqslant
\def\setN{\mathbb{N}}
\def\setR{\mathbb{R}}
\def\calC{\mathcal{C}}
\def\calF{\mathcal{F}}
\def\calL{\mathcal{L}}
\def\calT{\mathcal{T}}
\def\Circle{\mathfrak{C}}
\def\Line{\mathfrak{L}}

\title{A solution to Ringel's circle problem}

\author{James Davies}{Department of Pure Mathematics and Mathematical Statistics, University of Cambridge, UK}{jgd37@cam.ac.uk}{}{}
\author{Chaya Keller}{School of Computer Science, Ariel University, Israel}{chayak@ariel.ac.il}{}{}
\author{Linda Kleist}{Department of Computer Science, Technische Universität Braunschweig, Germany}{kleist@ibr.cs.tu-bs.de}{}{}
\author{Shakhar Smorodinsky}{Department of Mathematics, Ben-Gurion University of the Negev, Be'er-Sheva, Israel}{shakhar@math.bgu.ac.il}{}{}
\author{Bartosz Walczak}{Department of Theoretical Computer Science, Jagiellonian University in Kraków, Poland}{bartosz.walczak@uj.edu.pl}{}{}

\authorrunning{J. Davies, C. Keller, L. Kleist, S. Smorodinsky, B. Walczak}

\funding{Chaya Keller and Shakhar Smorodinsky were partially funded by the Israel Science Foundation (grant no.~1065/20).
Bartosz Walczak was partially supported by the National Science Center of Poland grant 2019/34/E/ST6/00443.}

\begin{document}

\maketitle

\begin{abstract}
We construct families of circles in the plane such that their tangency graphs have arbitrarily large girth and chromatic number.
This provides a strong negative answer to Ringel's circle problem (1959).
The proof relies on a (multidimensional) version of Gallai's theorem with polynomial constraints, which we derive from the Hales-Jewett theorem and which may be of independent interest.
\end{abstract}

\section{Introduction}

A \emph{constellation} (see~\cite{JR84}) is a finite collection of circles in the plane in which no three circles are tangent at the same point.
The \emph{tangency graph} $G(\calC)$ of a constellation $\calC$ is the graph with vertex set $\calC$ and edges comprising the pairs of tangent circles in $\calC$.
In this paper, graph-theoretic terms such as chromatic number or girth (i.e., the minimum length of a cycle) applied to a constellation $\calC$ refer to the tangency graph $G(\calC)$.

\begin{figure}[ht]
    \centering
    \includegraphics[page=1]{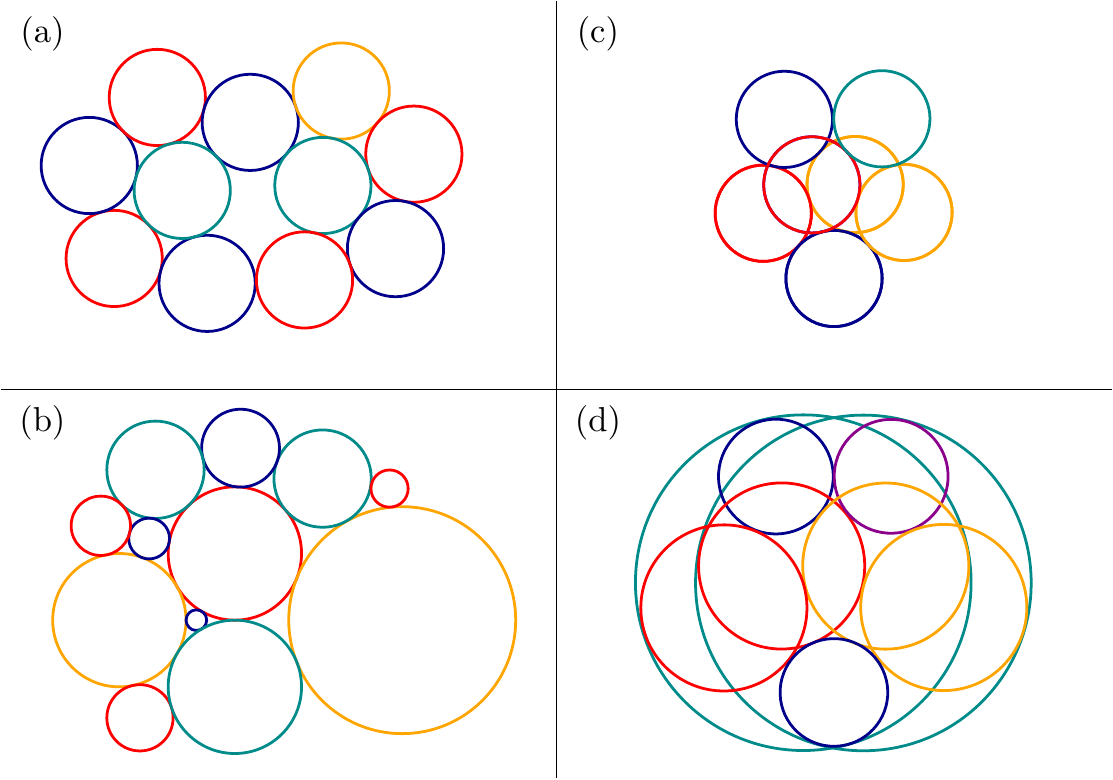}
    \caption{An illustration of the four coloring problems of tangency graphs of constellations: (a)~a~penny graph, (b) a coin graph, (c) an overlapping penny graph, and (d) a general constellation as in the circle problem.}
    \label{fig:examples}
\end{figure}

Jackson and Ringel~\cite{JR84} discussed four problems regarding the chromatic number of constellations.
The problems are illustrated in~\cref{fig:examples}.

\begin{alphaenumerate}
    \item \emph{The penny problem}. What is the maximum chromatic number of a constellation of non-overlapping unit circles?
    
    \item \emph{The coin problem}. What is the maximum chromatic number of a constellation of non-overlapping circles (of arbitrary radii)?

    \item \emph{The overlapping penny problem}. What is the maximum chromatic number of a (possibly overlapping) constellation of unit circles?
    
    \item \emph{The circle problem}. What is the maximum chromatic number of a general constellation of circles?
\end{alphaenumerate}

Jackson and Ringel provided a simple proof that the answer to the \emph{penny problem} is~$4$.
The claim that the answer for the \emph{coin problem} is also $4$ is equivalent to the \emph{four color theorem}~\cite{AH77,AHK77}.
Indeed, on the one hand, if the circles are non-overlapping, then $G(\calC)$ is planar and thus $4$-colorable by the four-color theorem.
On the other hand, as was observed by Sachs~\cite{Sachs94}, by the Koebe-Andreev-Thurston \emph{circle packing theorem}~\cite{Koebe36}, every planar graph can be realized as $G(\calC)$ for some constellation $\calC$ of non-overlapping circles, and hence, the assertion that every such constellation $\calC$ is $4$-colorable implies the four color theorem.

The \emph{overlapping penny problem} is equivalent to the celebrated Hadwiger-Nelson problem, which asks what is the minimum number of colors needed for a coloring of the plane such that no two points at distance $1$ get the same color.
Indeed, if all circles in $\calC$ have a radius of~$\nicefrac{1}{2}$, then two circles are tangent if and only if the distance between their centers is $1$.
For this setting, Isbell observed about 60 years ago that $7$ colors suffice (see~\cite{Soifer09}), and only recently de Grey~\cite{AdG18} showed that $4$ colors are not sufficient, and hence, the chromatic number of the plane lies between $5$ and $7$.
(Note that although the overlapping penny problem considers only finite sets of circles and the Hadwiger-Nelson problem considers an infinite set, they are still equivalent by a standard compactness argument.)

Unlike for the first three problems, in which a finite upper bound was known already when they were stated, for the \emph{circle problem} no finite upper bound was known.
This open problem was introduced for the first time by Ringel~\cite{Ringel59} in 1959 and appeared in several places as either a question (e.g., \cite{JR84,JT95,Pach17}) or a conjecture that there is a finite upper bound (e.g., \cite{Kalai15}).
For lower bounds, Jackson and Ringel~\cite{JR84} presented an example that requires $5$ colors; see \cref{fig:examples}(d).
Another such example follows from de Grey's $5$-chromatic unit distance graph.
No construction requiring more than $5$ colors has been known so far.

In this paper, we solve Ringel's circle problem in a strong sense by showing that the chromatic number is unbounded, even if we require high girth.

\begin{theorem}\label{main}
There exist constellations of circles in the plane with arbitrarily large girth and chromatic number.
\end{theorem}

The constellation condition (that no three circles are tangent at a point) is crucial for Ringel's circle problem to be interesting---otherwise one could drive the chromatic number arbitrarily high by taking a set of circles all tangent at one point.
In \cref{main}, however, the condition is redundant because it follows from the stronger condition that the girth of the tangency graph is greater than $3$.
Actually, we prove an even stronger statement (\cref{main induction}) in which we additionally forbid pairs of internally tangent circles.

The first author~\cite{Davies21} recently proved that there are intersection graphs of axis-aligned boxes in $\setR^3$ with arbitrarily large girth and chromatic number.
The main tool for this result is a ``sparse'' version of Gallai's theorem due to Prömel and Voigt~\cite{PV90} (see \cref{Gallai}),
which was applied in a modification of Tutte's construction of triangle-free graphs with large chromatic number~\cite{Descartes47,Descartes54}.

To prove \cref{main}, we also use  a ``sparse'' version of Gallai's theorem.
However, it is crucial in our context to guarantee that there are no ``unwanted'' tangencies in the resulting collection of circles.
To this end, we develop a refined ``sparse'' version of Gallai's theorem with additional (polynomial) constraints (\cref{Gallai v2}).
We believe that this version may be applicable to obtaining lower bound constructions for other geometric coloring problems, in which some specific form of algebraic independence is requested.

Tangent circles can be thought of as circles intersecting at zero angle.
We extend \cref{main} to graphs defined by pairs of circles intersecting at an arbitrary fixed angle.
Specifically, we say that two intersecting circles $C_1$ and $C_2$ intersect at angle $\theta$ if at any intersection point of $C_1$ and $C_2$, the (smaller) angle between the tangent line to $C_1$ and the tangent line to $C_2$ equals $\theta$.
For any $\theta\in[0,\nicefrac{\pi}{2}]$, the \emph{$\theta$-graph} $G_\theta(\calC)$ of a collection of circles $\calC$ is the graph with vertex set~$\calC$ and edges comprising the pairs of circles in $\calC$ that intersect at angle~$\theta$.
In particular, the $0$-graph is the tangency graph.
We extend \cref{main} as follows.

\begin{theorem}\label{theta}
For every\/ $\theta\in[0,\nicefrac{\pi}{2}]$, there exist\/ $\theta$-graphs of circles in the plane with arbitrarily large girth and chromatic number.
\end{theorem}

The proof of \cref{theta} for $\theta>0$ is significantly simpler than the proof for $\theta=0$ corresponding to \cref{main}.
Therefore, the remainder of the paper is organized as follows.
In \cref{sec:gallai}, we introduce Gallai's theorem and prove a version of it with additional constraints as needed for the proof of \cref{main}.
In \cref{sec:theta}, we prove \cref{theta} for $\theta>0$.
As the underlying ideas and tools are similar but simpler, this can be considered as a warm-up for the proof of the more involved case $\theta=0$, which follows in \cref{sec:main}.
Finally, in \cref{sec:Gallai revisited}, we provide an even stronger, ``induced'' version of Gallai's theorem with constraints.

\section{Gallai's theorem with constraints}
\label{sec:gallai}

We start by introducing results from Ramsey theory---Gallai's theorem and its versions that we need for the proofs of \cref{main,theta}.

A \emph{homothetic map} in $\setR^d$ is a map $h\colon\setR^d\to\setR^d$ of the form $h(p)=p^*+\lambda p$ for some~$p^*\in\setR^d$ and $\lambda>0$.
In other words, a homothetic map is a composition of (positive) uniform scaling and translation.
A set $T'\subseteq\setR^d$ is a \emph{homothetic copy} of a set $T\subseteq\setR^d$ if there is a homothetic map $h$ in $\setR^d$ such that $T'=h(T)$.

The following beautiful theorem, which is a generalization of the well-known van der Waerden's theorem on arithmetic progressions~\cite{VanderWaerden27}, was first discovered by Gallai in the~1930s, as reported by Rado~\cite{Rado45}.

\begin{gallai}
For every finite set\/ $T\subset\setR^d$ and every positive integer\/ $k$, there exists a finite set\/ $X\subset\setR^d$ such that every\/ $k$-coloring of\/ $X$ contains a monochromatic homothetic copy of\/ $T$.
\end{gallai}

A \emph{cycle} of length $\ell\geq 2$ on a set $X$ is a tuple $(T_1,\ldots,T_\ell)$ of distinct subsets of $X$ such that there exist distinct elements $x_1,\ldots,x_\ell\in X$ with $x_i\in T_i\cap T_{i+1}$ for $i\in[\ell-1]$ and $x_\ell\in T_\ell\cap T_1$.

In order to guarantee high girth in the proofs of \cref{main,theta}, we need an appropriate ``sparse'' version of Gallai's theorem, which excludes short cycles among the claimed homothetic copies of $T$ in $X$ (one of which is guaranteed to be monochromatic).
In particular, the following strengthening of Gallai's theorem suffices for the purpose of proving \cref{theta} for $\theta>0$.

\begin{theorem}[Prömel, Voigt~\cite{PV90}]\label{Gallai}
For every finite set\/ $T\subset\setR^d$ of size at least\/ $3$ and for any integers\/ $g\geq 3$ and\/ $k\geq 1$, there exists a finite set\/ $X\subset\setR^d$ such that every\/ $k$-coloring of\/ $X$ contains a monochromatic homothetic copy of\/ $T$ and no tuple of fewer than\/ $g$ homothetic copies of\/ $T$ in\/ $X$ forms a cycle on\/ $X$.
\end{theorem}

Since \cref{Gallai} only guarantees the existence of a set $X$, it is not specific enough to prove \cref{main}.
Roughly speaking, in our proof of \cref{main}, we apply (a refined version of) Gallai's theorem to a family of circles in the plane (with $d=3$, the third coordinate representing the radius) such that the resulting family of circles satisfies a number of additional conditions, e.g., it does not contain two internally tangent circles.
To guarantee the additional properties, we develop a refined ``sparse'' version of Gallai's theorem, which imposes polynomial constraints on the resulting set.

We say that a family $\calF$ of $2d$-variate real polynomials \emph{respects} a set $X\subset\setR^d$ if $f(p,q)\neq 0$ for all $f\in\calF$ and all pairs of distinct points $p,q\in X$.

\begin{theorem}\label{Gallai v2}
Let\/ $T$ be a finite subset of\/ $\setR^d$ of size at least\/ $3$, let\/ $\calF$ be a countable family of\/ $2d$-variate real polynomials that respects\/ $T$, and let\/ $g$ and\/ $k$ be positive integers.
Then there exist a finite set\/ $X\subset\setR^d$ and a collection\/ $\calT$ of homothetic copies of\/ $T$ in\/ $X$ satisfying the following conditions:
\begin{enumerate}
\item $\calF$ respects\/ $X$,
\item no tuple of fewer than\/ $g$ homothetic copies of\/ $T$ in\/ $\calT$ form a cycle,
\item every\/ $k$-coloring of\/ $X$ contains a monochromatic homothetic copy of\/ $T$ in\/ $\calT$.
\end{enumerate}
\end{theorem}

While \cref{Gallai v2} is what we need for the proof of \cref{main}, it is weaker than \cref{Gallai} in one aspect---its assertion holds for some collection $\calT$ of homothetic copies of $T$ rather than for all homothetic copies.
In \cref{sec:Gallai revisited}, we provide a common generalization of \cref{Gallai,Gallai v2} where $\calT$ can be taken as the set of \emph{all} homothetic copies of $T$ in $X$.

One of the ways of proving Gallai's theorem is to derive it from the Hales-Jewett theorem~\cite{HJ63}.
Our proof of \cref{Gallai v2} goes along the same line.

For $m,n\in\setN$, a subset $L$ of the $n$-dimensional $m$-cube $[m]^n$ is called a \emph{combinatorial line} if there exist a non-empty set of indices $I=\{i_1,\ldots,i_k\}\subseteq[n]$ and a choice of $x^*_i\in[m]$ for every $i\in[n]\setminus I$ such that
\[L=\bigl\{(x_1,\ldots,x_n)\in[m]^n:x_{i_1}=\cdots=x_{i_k}\text{ and }x_i=x^*_i\text{ for }i\notin I\bigr\}.\]
The indices in $I$ are called the \emph{active coordinates} of $L$.

\begin{halesjewett}
For any\/ $m,k\in\setN$, there exists\/ $n\in\setN$ such that every\/ $k$-coloring of\/ $[m]^n$ contains a monochromatic combinatorial line.
\end{halesjewett}

We need the following ``sparse'' version of the Hales-Jewett theorem.

\begin{theorem}[Prömel, Voigt~\cite{PV88}]\label{Hales}
For any\/ $m,g,k\in\setN$ with\/ $m\geq 3$, there exist\/ $n\in\setN$ and a set\/ $H\subseteq[m]^n$ such that every\/ $k$-coloring of\/ $H$ contains a monochromatic combinatorial line of\/ $[m]^n$ and no tuple of fewer than\/ $g$ combinatorial lines of\/ $[m]^n$ contained in\/ $H$ forms a cycle.
\end{theorem}

We also need the following simple algebraic fact.

\begin{lemma}\label{zero set}
For every countable family\/ $\calF$ of\/ $n$-variate real polynomials that are not identically zero, the union of their zero sets\/ $\bigcup_{f\in\calF}Z(f)$, where\/ $Z(f)=\{x\in\setR^n:f(x)=0\}$, has empty interior.
\end{lemma}

\begin{proof}
Fix $f\in\calF$.
Clearly, $Z(f)$ is a closed set in $\setR^n$.
First, observe that $Z(f)$ has empty interior.
Indeed, suppose to the contrary that there is a point $x$ in the interior of $Z(f)$.
Let $y\in\setR^n$ be such that $f(y)\neq 0$.
The univariate polynomial $f_{x,y}$ given by $f_{x,y}(t)=f(x+t(y-x))$ is not identically zero, because $f_{x,y}(1)=f(y)\neq 0$, so it has finitely many roots.
However, $f_{x,y}(t)=0$ whenever $|t|$ is sufficiently small for the point $x+t(y-x)$ to fall into an open neighborhood of $x$ contained in $Z(f)$.
There are infinitely many such values $t$, which is a contradiction.
Hence, $Z(f)$ has empty interior.
The lemma now follows by the Baire category theorem---a standard tool from topology, which asserts that a countable union of closed sets with empty interior in a complete metric space (such as $\setR^n$ with the Euclidean metric) has empty interior.
\end{proof}

Now, we are ready to prove \cref{Gallai v2}.

\begin{proof}[Proof of \cref{Gallai v2}]
We can assume without loss of generality that $\calF$ contains the polynomial $\delta$ defined by
\[\delta(p_1,\ldots,p_d,q_1,\ldots,q_d)=(p_1-q_1)^2+\cdots+(p_d-q_d)^2,\]
because distinct points $p,q\in\setR^d$ satisfy $\delta(p,q)\neq 0$.
Put $T=\{t_1,\dots,t_m\}$, where $m=|T|$.
Let $n$ and $H\subseteq[m]^n$ be as claimed in \cref{Hales} applied to $[m]$, $g$, and $k$.
For a given vector $\gamma=(\gamma_1,\ldots,\gamma_n)\in \setR^n$, define a map $\zeta_\gamma\colon H\to\setR^d$ by $\zeta_\gamma(x)=\sum_{i=1}^n\gamma_it_{x_i}$, and put $X_\gamma=\zeta_\gamma(H)=\{\zeta_\gamma(x):x\in H\}\subset\setR^d$.

We aim to find a vector $\gamma\in\setR^n$ with positive coordinates such that $\calF$ respects the set $X_\gamma$.
For any $f\in\calF$ and any distinct $x,y\in H$, let $F_{f,x,y}$ be the $n$-variate polynomial defined by
\[\textstyle F_{f,x,y}(\gamma_1,\ldots,\gamma_n)=f\bigl(\sum_{i=1}^n\gamma_it_{x_i},\;\sum_{i=1}^n\gamma_it_{y_i}\bigr).\]
Given $f\in\calF$ and distinct points $x,y\in H$, let $i\in[n]$ be an index such that $x_i\neq y_i$.
Setting $\gamma_i=1$ and $\gamma_j=0$ for $j\neq i$, we obtain $F_{f,x,y}(\gamma)=f(t_{x_i},t_{y_i})\neq 0$, which shows that $F_{f,x,y}$ is not identically zero.
Apply \cref{zero set} to the family $\{F_{f,x,y}:f\in\calF$, $x,y\in H$, $x\neq y\}$ to conclude that the union $Z=\bigcup_{f\in\calF,\,x,y\in H,\,x\neq y}Z(F_{f,x,y})$ of the zero sets of the polynomials~$F_{f,x,y}$ has empty interior, that is, the set $U\setminus Z$ is non-empty for every non-empty open set $U\subseteq\setR^n$.
In particular, there is a vector $\gamma$ with positive coordinates that is not in~$Z$.
Thus $F_{f,x,y}(\gamma)\neq 0$ for all $f\in\calF$ and all distinct $x,y\in H$, so $\calF$ respects the set $X_\gamma$.

Fix such a vector $\gamma$, and let $\zeta=\zeta_\gamma$ and $X=X_\gamma$.
Condition~1 thus follows.
Furthermore, by our assumption that $\delta\in\calF$, we have $\delta(\zeta(x),\zeta(y))=F_{\delta,x,y}(\gamma)\neq 0$ for any distinct $x,y\in H$, which shows that $\zeta$ is injective.

Let $\calL$ be the set of combinatorial lines that are contained in $H$.
Every combinatorial line $L\in\calL$ gives rise to a homothetic copy of $T$ in $X$ as follows: if $I$ is the set of active coordinates of $L$ and the coordinates $i\notin I$ are fixed to $x_i$ in $L$, then the set
\[\textstyle\zeta(L)=\{\zeta(x):x\in L\}=\bigl\{\sum_{i\notin I}\gamma_it_{x_i}+\bigl(\sum_{i\in I}\gamma_i\bigr)t_j:j\in[m]\bigr\}\]
is a homothetic copy of $T$.
Specifically, we have $\zeta(L)=h(T)$ for the homothetic map $h$ given by $h(p)=p^*+\lambda p$ with $p^*=\sum_{i\notin I}\gamma_it_{x_i}$ and $\lambda=\sum_{i\in I}\gamma_i>0$.
Let $\calT=\{\zeta(L):L\in\calL\}$.

We show that conditions 2 and~3 hold for $X$ and $\calT$.
Since no tuple of fewer than $g$ combinatorial lines in $\calL$ form a cycle and $\zeta$ is injective, no tuple of fewer than $g$ members of $\calT$ form a cycle, which is condition~2.
For the proof of condition~3, consider a $k$-coloring~$\phi$ of $X$.
It induces a $k$-coloring $x\mapsto\phi(\zeta(x))$ of $H$, in which, by \cref{Hales}, there is a monochromatic combinatorial line $L\in\calL$.
We conclude that the homothetic copy $\zeta(L)$ of $T$ in $\calT$ is monochromatic in $\phi$.
\end{proof}

The above proof method can also be used to prove other versions of Gallai's theorem with constraints.
For example, we can describe constraints using functions other than polynomials if they satisfy a suitable analogue of \cref{zero set}, e.g., real analytic functions (see~\cite{Mityagin20}).
In addition, we can use other versions of the Hales-Jewett theorem, e.g., the density Hales-Jewett theorem due to Furstenberg and Katznelson~\cite{FK91}, which asserts that for any $m\in\setN$ and $\alpha>0$, there is $n\in\setN$ such that every subset of $[m]^n$ of size at least $\alpha m^n$ contains a combinatorial line.
Then, the same proof leads to the following result.

\begin{theorem}
Let\/ $T$ be a finite subset of\/ $\setR^d$, let\/ $\calF$ be a countable family of\/ $2d$-variate real analytic functions that respects\/ $T$, and let\/ $\alpha>0$.
Then there exists a finite set\/ $X\subset\setR^d$ such that\/ $\calF$ respects\/ $X$ and every subset of\/ $X$ of size at least\/ $\alpha|X|$ contains a homothetic copy of\/ $T$.
\end{theorem}

\section{Proof of Theorem \ref{theta} for \texorpdfstring{$\theta>0$}{θ>0}}
\label{sec:theta}

In this section, we prove \cref{theta} for all $\theta\in(0,\nicefrac{\pi}{2}]$.
Here is the precise statement.

\begin{theorem}\label{theta not 0}
For every\/ $\theta\in(0,\nicefrac{\pi}{2}]$ and any integers\/ $g\geq 3$ and\/ $k\geq 1$, there is a collection of circles\/ $\calC$, no two concentric, such that the\/ $\theta$-graph\/ $G_\theta(\calC)$ has girth at least\/ $g$ and chromatic number at least\/ $k$.
\end{theorem}

\begin{proof}
We fix $\theta\in(0,\nicefrac{\pi}{2}]$ and $g\geq 3$, and construct families of circles by induction on $k$.
For the base case $k\leq 3$, we represent odd cycles as $\theta$-graphs of circles.
To represent the cycle~$C_n$, consider a regular $n$-gon of sidelength $s\in(\sqrt{2},2]$, and place unit circles with centers at its corners; see~\cref{fig:warmUp3}.
\begin{figure}[b]
    \centering
    \includegraphics[page=8]{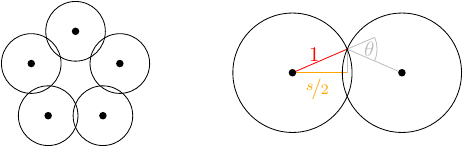}
    \caption{In a regular $n$-gon of sidelength $s$, the circles intersect at an angle of $\theta=\pi-2\arcsin(\nicefrac{s}{2})$.}
    \label{fig:warmUp3}
\end{figure}
Then circles at neighboring corners intersect at an angle of $\theta=\pi-2\arcsin(\nicefrac{s}{2})$, which continuously changes from $\nicefrac{\pi}{2}$ for $s=\sqrt{2}$ to $0$ for $s=2$; here we exploit the fact that the tangents at the intersection points are orthogonal to the radii.
Additionally, it is easy to verify that no other pairs of circles intersect.
Hence, for an appropriate choice of~$s$, the $\theta$-graph of the family of circles is $C_n$.

For the induction step, assume we have already constructed a family of circles $\calC_k$, no two concentric, such that the $\theta$-graph $G_\theta(\calC)$ has girth at least $g$ and chromatic number at least $k$, where $k\geq 3$.
Let $\Line(y)$ denote the horizontal line at coordinate $y\in\setR$, that is, $\Line(y)=\setR\times\{y\}$.
To construct a family $\calC_{k+1}$, we perform the following process.

\begin{enumerate}
    \item \emph{Constructing a ``template'' set from the family\/ $\calC_k$.}
    For each circle $C\in\calC_k$, we pick all horizontal lines that intersect $C$ at angle~$\theta$ (meaning that the angle between the horizontal line and the tangent line to $C$ at either of the intersection points is $\theta$); see \cref{fig:warmUp1}.
    There are two such lines when $\theta\in(0,\nicefrac{\pi}{2})$ and only one (through the center of the circle) when $\theta=\nicefrac{\pi}{2}$.
    By slightly rotating the family $\calC$ if needed, we can guarantee that these lines are all distinct.
    When $\theta=\nicefrac{\pi}{2}$, this requires the additional assumption that no two circles in $\calC$ are concentric (which is otherwise superfluous).
    Let $T_0\subset\setR$ be the set of $y$-coordinates of the lines.
    \begin{figure}[htb]
        \centering
        \includegraphics[page=2]{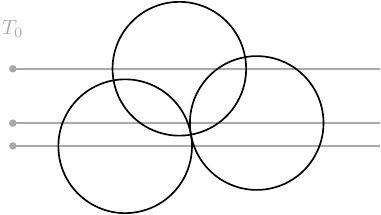}
        \caption{Construction of the set $T_0$ for $\theta=\nicefrac{\pi}{2}$.}
        \label{fig:warmUp1}
    \end{figure}

    \item \emph{Applying Gallai's theorem.}
    \cref{Gallai} in $\setR$ applied to the set $T_0$ yields a finite set $Y\subset\setR$ such that every $k$-coloring of $Y$ contains a monochromatic homothetic copy of $T_0$ and no tuple of fewer than $\lceil\nicefrac{g}{2}\rceil$ homothetic copies of $T_0$ in $Y$ form a cycle.
    
    \item \emph{Geometric interpretation of the resulting set.}
    The set $Y$ gives rise to the family of horizontal lines $\calL'=\{\Line(y):y\in Y\}$.
    Let $\calT$ be the family of homothetic copies of $T_0$ in $Y$.
    
    \item \emph{Attaching a copy of\/ $\calC_k$ for each homothetic copy of the ``template''.}
    For each $T\in\calT$, we consider the set of horizontal lines $\calL'_T=\{\Line(y):y\in T\}$ and construct a homothetic copy~$\calC'_T$ of $\calC_k$ such that each line in $\calL'_T$ intersects a single circle $\calC'_T$ at angle $\theta$; see \cref{fig:warmUp2}.
    (Each circle in $\calC'_T$ intersects two lines in $\calL'_T$ at angle $\theta$ when $\theta\in(0,\nicefrac{\pi}{2})$ and only one when $\theta=\nicefrac{\pi}{2}$.)
    This is possible because the set of lines $\{\Line(y):y\in T_0\}$ has this property with respect to $\calC_k$, and $T$ is a homothetic copy of $T_0$.
    We spread the copies $\calC'_T$ horizontally so that a vertical line separates $\calC'_{T_1}$ from $\calC'_{T_2}$ for any distinct $T_1,T_2\in\calT$.
    Let $\calC'=\bigcup_{T\in\calT}\calC'_T$.

    \begin{figure}[htb]
        \centering
        \includegraphics[page=10]{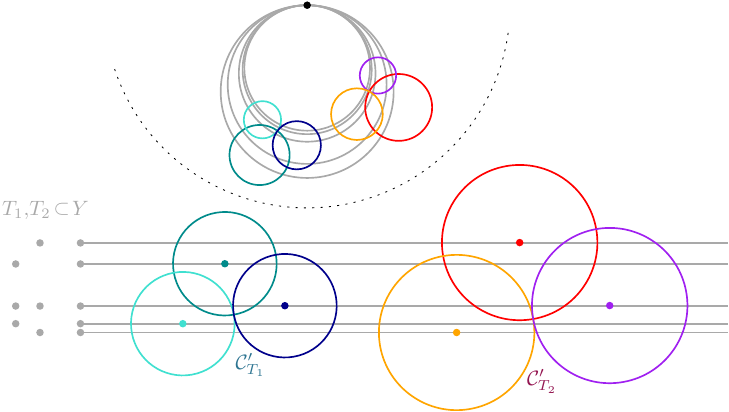}
        \caption{Illustration for steps 3--5 in the proof of \cref{theta not 0} for $\theta=\nicefrac{\pi}{2}$:
        Construction of a preliminary family from $Y$ and the result after inversion with respect to the dotted circle.
        Note that the set $Y$ consists of two homothetic copies of the set $T_0$ from \cref{fig:warmUp1} that have one element in common.}
        \label{fig:warmUp2}
    \end{figure}

    \item \emph{Constructing the final family\/ $\calC_{k+1}$ via inversion.}
    Finally, we construct the family $\calC_{k+1}$ by applying a geometric inversion to the lines and circles in $\calL'\cup\calC'$, where the center of inversion is chosen not to lie on any of these lines or circles.
    See \cref{fig:warmUp2} for an illustration.
    By basic properties of inversion, the resulting family consists only of circles~\cite[Chapter~6]{Coxeter61}; in particular, the lines in $\calL'$ turn into a family of circles that are all tangent at the center of inversion.
    To ensure that the inversion does not create concentric circles, we choose the center of inversion not to lie on any line passing through the centers of two circles in $\calC'$ or any vertical line passing through the center of a circle in $\calC'$.
\end{enumerate}

We claim that the $\theta$-graph $G_\theta(\calC_{k+1})$ has girth at least $g$ and chromatic number at least~$k+1$.
Since inversion preserves angles, this graph is isomorphic to the $\theta$-graph of $\calL'\cup\calC'$ (which is defined analogously to the $\theta$-graph for a collection of circles).
Let $G$ denote the latter $\theta$-graph.
It is thus sufficient to prove that $G$ has girth at least $g$ and chromatic number at least $k+1$.

To this end, we observe that by the construction, $G$ has the following structure: for every~$T\in\calT$, the subgraph of $G$ induced on the vertices in $\calC'_T$ is isomorphic to $G_\theta(\calC_k)$, and the remaining edges form a collection of bipartite subgraphs between the vertices in $\calC'_T$ and the vertices in $\calL'_T$, where each vertex in $\calC'_T$ is adjacent to two corresponding vertices in $\calL'_T$ if $\theta\in(0,\nicefrac{\pi}{2})$ and only one if $\theta=\nicefrac{\pi}{2}$.

We exploit the structure above in the proofs of the final two claims.
They are standard when applying generalizations of Tutte's construction; see, e.g., \cite{Davies21,KN99}.

\begin{subclaim}\label{claim:girth-theta}
The graph $G$ has girth at least $g$.
\end{subclaim}

\begin{claimproof}
For every $T\in\calT$, every cycle in $G$ that lies entirely within $\calC'_T$ has length at least $g$, because the subgraph of $G$ induced on the vertices in $\calC'_T$ is isomorphic to $G_\theta(\calC_k)$, the girth of which is at least $g$ by the induction hypothesis.
It thus remains to consider a cycle in $G$ of length $\ell\geq 3$ that does not lie entirely within $\calC'_T$ for any $T\in\calT$.
It must contain vertices from $\calL'$, say, $L_1,\ldots,L_m$ in this order along the cycle.
For each $i\in[m]$, since $L_i$ has no edges to the rest of $\calL'$ and at most one edge to $\calC'_T$ for each $T\in\calT$, the neighbors of $L_i$ on the cycle lie in two different sets of the form $\calC'_T$.
For each $i\in[m]$, let $T_i\in\calT$ be such that the part of the cycle between $L_i$ and $L_{i+1}$ (or $L_1$ when $i=m$) lies within $\calC'_{T_i}$.

It follows that $(T_1,\ldots,T_m)$ is a cycle in $\calT$ of length $m$ or contains such a smaller cycle if some members of $\calT$ repeat among $T_1,\ldots,T_m$.
Hence, \cref{Gallai} yields $m\geq\lceil\nicefrac{g}{2}\rceil$.
Since the cycle contains at least one vertex from $\calC'$ between $L_i$ and $L_{i+1}$ (or $L_1$ when $i=m$) for any $i\in[m]$, we conclude that $\ell\geq 2m\geq g$.
\end{claimproof}

\begin{subclaim}\label{claim:chromatic number-theta}
The graph $G$ has chromatic number at least $k+1$.
\end{subclaim}

\begin{claimproof}
Suppose for the sake of contradiction that the graph $G$ is $k$-colorable.
Pick a proper $k$-coloring of $G$, and consider its restriction to the vertices in $\calL'$.
It induces a $k$-coloring of $Y$ via the correspondence $Y\ni y\leftrightarrow\Line(y)\in\calL'$.
It follows from the application of \cref{Gallai} that there is a monochromatic homothetic copy $T$ of $T_0$ in $\calT$, which means that the set of lines $\calL'_T$ is monochromatic.
Since the edges of $G$ that connect these lines with $\calC'_T$ match all of $\calC'_T$, their common color does not occur on the circles in $\calC'_T$.
Therefore, the given $k$-coloring of $G$ induces a proper $(k-1)$-coloring of the subgraph of $G$ induced on the vertices in $\calC'_T$, which is isomorphic to $G_\theta(\calC_k)$.
This contradicts the assumption that the graph $G_\theta(\calC_k)$ has chromatic number at least $k$.
\end{claimproof}

This completes the proof of \cref{theta not 0} by induction.
\end{proof}

Observe that the proof above cannot be used for $\theta=0$ to prove \cref{main}, because the inversion at step~5 turns all lines in $\calL'$ into circles tangent at one point, so the resulting collection of circles is not a constellation (and the resulting tangency graph has girth $3$).

\section{Proof of Theorem \ref{main}}
\label{sec:main}

In this section, we prove \cref{main}.
For the purpose of induction, we prove the following stronger statement, which directly implies \cref{main}.

\begin{theorem}\label{main induction}
For any integers\/ $g\geq 3$ and\/ $k\geq 1$, there exists a collection of circles\/ $\calC$, no two concentric and no two internally tangent, such that the tangency graph\/ $G(\calC)$ has girth at least\/ $g$ and chromatic number at least\/ $k$.
\end{theorem}

For the sake of clarity, we first present the construction of the families $\calC$ and then we prove that the construction satisfies the requirements of the theorem.

\subsection{High-level description of the construction}

We fix $g\geq 3$ and prove the theorem by induction on $k$.
The base case $k\leq 3$ is easy, because all odd cycles can be represented as tangency graphs of circles satisfying the conditions of the theorem.
For instance, a cycle $C_n$ can be represented by unit circles whose centers are placed at the corners of a regular $n$-gon of sidelength $2$.

For the induction step, assume we have already constructed a family of circles $\calC_k$, no two concentric and no two internally tangent, such that the tangency graph $G(\calC_k)$ has girth at least $g$ and chromatic number at least $k$.
Let $\Circle(x,y,r)$ denote the circle with center $(x,y)\in\setR^2$ and radius $r>0$.
To construct a family $\calC_{k+1}$ with girth $g$ and chromatic number $k+1$, we perform the following process.
See \cref{fig:constructionMain} for an illustration.

\begin{figure}[htb]
    \centering
    \includegraphics[page=9]{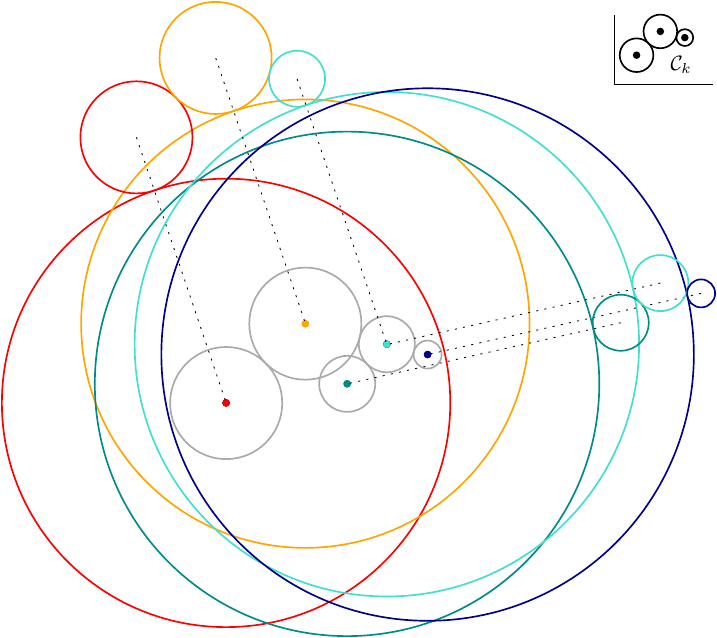}
    \caption{Illustration for the construction of the family $\calC_{k+1}$ from $\calC_k$.
    Gray circles represent a part of the set $X\subset\setR^3$ containing two homothetic copies of $T_0$ in $\calT$, both such that $r^*_T=0$ (for clarity of the illustration).
    The family $\calC_{k+1}$ contains a (large) circle for each point in $X$ and a homothetic copy of $\calC_k$ for each homothetic copy of $T_0$ in $\calT$.}
    \label{fig:constructionMain}
\end{figure}

\begin{enumerate}
    \item \emph{Constructing a ``template'' set from the family\/ $\calC_k$.}
    We represent each circle $\Circle(x,y,r)\in\calC_k$ by the point $(x,y,r)\in\setR^3$, to obtain the set
    \[T_0=\bigl\{(x,y,r)\in\setR^3:\Circle(x,y,r)\in\calC_k\bigr\}\subset\setR^3.\]
    
    \item \emph{Applying Gallai's theorem with constraints.}
    \cref{Gallai v2} in $\setR^3$ applied to the set $T_0$ with appropriate constraints to be detailed below yields a finite set $X\subset\setR^3$ and a collection~$\calT$ of homothetic copies of $T_0$ in $X$ such that every $k$-coloring of $X$ contains a monochromatic homothetic copy of $T_0$ in~$\calT$ and no tuple of fewer than $\lceil\nicefrac{g}{3}\rceil$ homothetic copies of $T_0$ in~$\calT$ form a cycle.
    For each $T\in\calT$, let $h_T$ be the homothetic map from $T_0$ to $T$ in $\setR^3$.
    It has the following form for some $x^*_T,y^*_T,r^*_T\in\setR$ and $\lambda_T>0$:
    \[h_T\colon\setR^3\ni(x,y,r)\mapsto\bigl(x^*_T+\lambda_Tx,\;y^*_T+\lambda_Ty,\;r^*_T+\lambda_Tr\bigr)\in\setR^3.\]
    
    \item \emph{Geometric interpretation of the resulting set.}
    Let $R_0=\max\{r':(x',y',r')\in X\}$, and let~$R\in\setR$ satisfy $R>R_0$.
    (In the sequel, $R$ is going to be ``large''.)
    The set $X$ gives rise to a family of ``large'' circles $\calC'_R$, parameterized by $R$, defined as follows:
    \[\calC'_R=\bigl\{\Circle(x',y',R-r'):(x',y',r')\in X\bigr\}.\]
    The use of the stronger \cref{Gallai v2} with appropriate constraints, rather than \cref{Gallai}, allows us to infer that no two circles in $\calC'_R$ are concentric or internally tangent.
    
    \item \emph{Attaching a copy of\/ $\calC_k$ for each homothetic copy of the ``template''.}
    We pick a set $\{\phi_T\}_{T\in\calT}$ of distinct angles in $[0,\pi)$ that satisfy a certain condition to be detailed below.
    For every $T\in\calT$ and every circle $C=\Circle(x,y,r)\in\calC_k$, we define the following two circles, where $(x',y',r')=h_T(x,y,r)\in T$:
    \begin{align*}
    \mu_{R,T}(C)&=\Circle(x',y',R-r'),\quad\text{which is a circle in }\calC'_R,\\
    \nu_{R,T}(C)&=\Circle\bigl(x'+(R-r^*_T)\cos(\phi_T),\;y'+(R-r^*_T)\sin(\phi_T),\;\lambda_Tr\bigr).
    \end{align*}
    In words, $\mu_{R,T}(C)$ is a ``large'' circle with center $(x',y')$, and $\nu_{R,T}(C)$ is a ``small'' circle with center translated from $(x',y')$ in direction $\phi_T$, externally tangent to $\mu_{R,T}(C)$.
    For every $T\in\calT$, we set \[\calC'_{R,T}=\{\nu_{R,T}(C):C\in\calC_k\}.\]
    Since the angles $\phi_T$ are distinct and the radii of the circles $\nu_{R,T}(C)$ do not depend on $R$, when $R$ is sufficiently large, the circles $\calC'_{R,T_1}$ are disjoint from those in $\calC'_{R,T_2}$ for any distinct $T_1,T_2\in\calT$.
    
    \item \emph{Constructing the final family\/ $\calC_{k+1}$.}
    Finally, we define \[\calC''_R=\calC'_R\cup\bigcup_{T\in\calT}\calC'_{R,T}.\]
    We will show that $\calC_{k+1}:=\calC''_R$ satisfies all claimed properties if $R$ is sufficiently large.
\end{enumerate}

\paragraph*{Comparison with the construction in Section~\ref{sec:theta}}

This construction follows the general strategy of the construction for the $\theta$-graph presented in \cref{sec:theta}.
Notable differences result from the need to avoid multiple circles mutually tangent at one point, which arise in the last step of that construction when applying inversion.

\begin{enumerate}
    \item While in \cref{sec:theta}, we have $T_0\subset\setR$, here we have to resort to the more complex choice of $T_0\subset\setR^3$.
    
    \item While in \cref{sec:theta}, the standard ``sparse'' version of Gallai's theorem is sufficient, here we need the stronger version with constraints, to be able to infer that the resulting set of circles avoids concentricities and internal tangencies.
    
    \item The construction of ``large'' circles here is explicit and uses a parameter $R$ that must be chosen appropriately, while in \cref{sec:theta}, horizontal lines are used instead.
    
    \item The construction of ``small'' circles here is more involved and uses a set of parameters $\{\phi_T\}_{T\in\calT}$ that must be chosen appropriately.
    
    \item In \cref{sec:theta}, a final application of inversion is required to transform the horizontal lines into circles, which is no longer needed in the construction here.
\end{enumerate}

\noindent
The proof of validity of the construction is somewhat more complex, accordingly.

\subsection{Proof of Theorem \ref{main induction}}

In the proof, we use the following two simple lemmas in addition to \cref{Gallai v2}.

\begin{lemma}\label{tangency conditions}
Circles\/ $\Circle(x_1,y_1,r_1)$ and\/ $\Circle(x_2,y_2,r_2)$ are
\begin{itemize}
\item externally tangent if and only if\/ $(x_1-x_2)^2+(y_1-y_2)^2=(r_1+r_2)^2$,
\item internally tangent if and only if\/ $(x_1-x_2)^2+(y_1-y_2)^2=(r_1-r_2)^2$.
\end{itemize}
\end{lemma}

\begin{proof}
Circles $\Circle(x_1,y_1,r_1)$ and $\Circle(x_2,y_2,r_2)$ are externally tangent if and only if the segment connecting their centers $(x_1,x_2)$ and $(y_1,y_2)$ has length $r_1+r_2$, and they are internally tangent if and only if it has length $|r_1-r_2|$.
See \cref{fig:tangencyConditions} for an illustration.
\begin{figure}[ht]
    \centering
    \includegraphics[page=5]{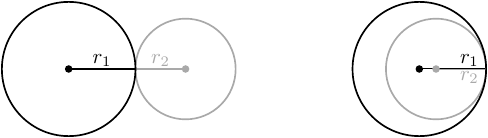}
    \caption{External and internal tangency of circles.}
    \label{fig:tangencyConditions}
\end{figure}
\end{proof}

\begin{lemma}\label{tangency polynomial}
If\/ $a,b,c,\varphi\in\setR$, $(a,b)\neq(0,0)$, and the vectors\/ $(a,b)$ and\/ $(\cos\varphi,\sin\varphi)$ are not parallel in\/ $\setR^2$, then the equality\/ $(a+R\cos\varphi)^2+(b+R\sin\varphi)^2=(c+R)^2$ holds for at most one value\/ $R\in\setR$.
\end{lemma}

\begin{proof}
Consider the univariate polynomial $f$ defined by
\begin{align*}
f(R)&=(a+R\cos\varphi)^2+(b+R\sin\varphi)^2-(c+R)^2\\
&=(a^2+b^2-c^2)+2R(a\cos\varphi+b\sin\varphi-c).
\end{align*}
It is identically zero only if $a^2+b^2=c^2$ and $a\cos\varphi+b\sin\varphi=c$.
However, if $a^2+b^2=c^2$ and the vectors $(a,b)$ and $(\cos\varphi,\sin\varphi)$ are not parallel, then the Cauchy-Schwarz inequality yields $|a\cos\varphi+b\sin\varphi|<\sqrt{a^2+b^2}\cdot\sqrt{\cos^2\varphi+\sin^2\varphi}=|c|$, so $a\cos\varphi+b\sin\varphi\neq c$.
Since $f$ is not identically zero and has degree at most $1$, it has at most one root in $\setR$.
\end{proof}

We are now ready to present the details of the proof of \cref{main induction}.

\newcounter{savedtheorem}
\setcounter{savedtheorem}{\value{theorem}}
\setcounterref{theorem}{main induction}

\begin{proof}[Proof of \cref{main induction}]
We proceed by induction on $k$ and, for the induction step, construct the family of circles $\calC_{k+1}$ from a family of circles $\calC_k$ as was described above.
The following claim describes the properties of the set $X$ constructed by applying our enhanced version of Gallai's theorem, namely \cref{Gallai v2}, with appropriate polynomial constraints to~$T_0$.

\begin{subclaim}\label{claim:Gallai}
There exists a finite set $X\subset\setR^3$ and a collection $\calT$ of homothetic copies of $T_0$ in $X$ with the following properties:
\begin{enumerate}
\item for any two distinct points $(x_1,y_1,r_1),(x_2,y_2,r_2)\in X$, we have
\begin{enumerate}
\item $(x_1,y_1)\neq(x_2,y_2)$,
\item $(x_1-x_2)^2+(y_1-y_2)^2\neq(r_1-r_2)^2$,
\end{enumerate}
\item no tuple of fewer than $\lceil\nicefrac{g}{3}\rceil$ homothetic copies of $T_0$ in $\calT$ form a cycle,
\item every $k$-coloring of $X$ contains a monochromatic homothetic copy of $T_0$ in $\calT$.
\end{enumerate}
\end{subclaim}

\begin{claimproof}
Consider the following two $6$-variate polynomials:
\begin{align*}
f_{\mathsf{a}}(x_1,y_1,r_1,x_2,y_2,r_2)&=(x_1-x_2)^2+(y_1-y_2)^2,\\
f_{\mathsf{b}}(x_1,y_1,r_1,x_2,y_2,r_2)&=(x_1-x_2)^2+(y_1-y_2)^2-(r_1-r_2)^2.
\end{align*}
We have $f_{\mathsf{a}}(x_1,y_1,r_1,x_2,y_2,r_2)\neq 0$ if and only if $(x_1,y_1)\neq(x_2,y_2)$, which holds in particular for distinct points $(x_1,y_1,r_1),(x_2,y_2,r_2)\in T_0$, by the assumption that no two circles in~$\calC_k$ are concentric.
\cref{tangency conditions} and the assumption that no two circles in $\calC$ are internally tangent imply $f_{\mathsf{b}}(x_1,y_1,r_1,x_2,y_2,r_2)\neq 0$ for any distinct points $(x_1,y_1,r_1),(x_2,y_2,r_2)\in T_0$.
\cref{Gallai v2} applied to $T_0$, $\calF=\{f_{\mathsf{a}},f_{\mathsf{b}}\}$, $\lceil\nicefrac{g}{3}\rceil$, and $k$ directly yields the requested set $X$ and collection $\calT$.
\end{claimproof}

For the construction of the families of circles $\{\nu_{R,T}\}_{T\in\calT}$, we let $\{\phi_T\}_{T\in\calT}$ be a set of distinct angles in $[0,\pi)$ such that for every $T\in\calT$, the unit vector $(\cos\phi_T,\sin\phi_T)\in\setR^2$ is not parallel to the vector $(x_1-x_2,y_1-y_2)$ for any distinct points $(x_1,y_1,r_1),(x_2,y_2,r_2)\in X$, where the latter vector is non-zero, by condition~1a of \cref{claim:Gallai}.

We now claim that for a sufficiently large $R$, the family $\calC''_R$ defined above satisfies the following conditions on concentricity and tangency.

\begin{subclaim}\label{claim:edges}
The following holds when $R$ is sufficiently large:
\begin{enumerate}
\item no two circles in $\calC''_R$ are concentric,
\item no two circles in $\calC''_R$ are internally tangent,
\item a pair of circles in $\calC''_R$ is externally tangent if and only if it belongs to one of the two following types:
\begin{enumerate}
\item $\mu_{R,T}(C)$ and $\nu_{R,T}(C)$ for any $T\in\calT$ and any $C\in\calC_k$,
\item $\nu_{R,T}(C_1)$ and $\nu_{R,T}(C_2)$ for any $T\in\calT$ and any $C_1,C_2\in\calC_k$ that are externally tangent.
\end{enumerate}
\end{enumerate}
\end{subclaim}

\begin{claimproof}
First, consider two distinct circles $C'=\Circle(x',y',R-r')$ and $C''=\Circle(x'',y'',R-r'')$ in $\calC'_R$, where $(x',y',r'),(x'',y'',r'')\in X$.
By condition~1a of \cref{claim:Gallai}, the circles $C'$ and $C''$ are not concentric.
By \cref{tangency conditions}, the circles $C'$ and $C''$ are internally tangent if and only if $(x'-x'')^2+(y'-y'')^2=(r'-r'')^2$, which does not hold due to condition~1b of \cref{claim:Gallai}.
Also by \cref{tangency conditions}, the circles $C'$ and $C''$ are externally tangent if and only if $(x'-x'')^2+(y'-y'')^2=(2R-r'-r'')^2$, which does not hold when $R$ is sufficiently large.

Next, let $T\in\calT$, and consider two distinct circles $C'_1$ and $C'_2$ such that for $i\in[2]$, we have $C'_i=\Circle\bigl(x'_i+(R-r^*_T)\cos\phi_T,\;y'_i+(R-r^*_T)\sin\phi_T,\;\lambda_Tr_i\bigr)=\nu_{R,T}(C_i)$, where $C_i=\Circle(x_i,y_i,r_i)\in\calC_k$ and $h_T(x_i,y_i,r_i)=(x'_i,y'_i,r'_i)\in T$.
The assumption that the circles $C_1$ and $C_2$ are not concentric, that is, $(x_1,y_1)\neq(x_2,y_2)$, yields $(x'_1,y'_1)\neq(x'_2,y'_2)$, which implies that the circles $C'_1$ and $C'_2$ are not concentric.
By \cref{tangency conditions}, the circles $C'_1$ and $C'_2$ are internally tangent if and only if $(x'_1-x'_2)^2+(y'_1-y'_2)^2=\lambda_T^2(r_1-r_2)^2$, which is equivalent to $(x_1-x_2)^2+(y_1-y_2)^2=(r_1-r_2)^2$, which does not hold due to the assumption that $C_1$ and $C_2$ are not internally tangent.
Also by \cref{tangency conditions}, the circles $C'_1$ and $C'_2$ are externally tangent if and only if $(x'_1-x'_2)^2+(y'_1-y'_2)^2=\lambda_T^2(r_1+r_2)^2$, which is equivalent to $(x_1-x_2)^2+(y_1-y_2)^2=(r_1+r_2)^2$, which means that $C_1$ and $C_2$ are externally tangent.

Next, for two distinct $T_1,T_2\in\calT$, consider circles of the form $C'_1=\nu_{R,T_1}(C_1)$ and $C'_2=\nu_{R,T_2}(C_2)$, where $C_1,C_2\in\calC_k$.
Since $\phi_{T_1}\neq\phi_{T_2}$ and the radii of $C'_1$ and $C'_2$ are independent of $R$, the circles $C'_1$ and $C'_2$ are arbitrarily far apart as $R$ grows.
In particular, they are disjoint and not nested when $R$ is sufficiently large.

Finally, consider a circle $C'=\Circle(x',y',R-r')\in\calC'_R$, where $(x',y',r')\in X$, and a circle $C''=\Circle\bigl(x''+(R-r^*_T)\cos\phi_T,\;y''+(R-r^*_T)\sin\phi_T,\;\lambda_Tr\bigr)=\nu_{R,T}(C)$, where $C=\Circle(x,y,r)\in\calC_k$ and $h_T(x,y,r)=(x'',y'',r'')\in T$.
When $R>r^*_T$, since the vectors $(x''-x',y''-y')$ and $(\cos\phi_T,\sin\phi_T)$ are not parallel unless $(x',y')=(x'',y'')$, the circles $C'$ and $C''$ are not concentric.
By \cref{tangency conditions}, the circles $C'$ and $C''$ are externally tangent if and only if
\[\bigl(x''-x'+(R-r^*_T)\cos\phi_T\bigr)^2+\bigl(y''-y'+(R-r^*_T)\sin\phi_T\bigr)^2=(R-r'+\lambda_Tr)^2,\]
which is equivalent to
\[\bigl((x''-x')+R'\cos\phi_T\bigr)^2+\bigl((y''-y')+R'\sin\phi_T\bigr)^2=\bigl((r''-r')+R'\bigr)^2,\]
where $R'=R-r^*_T$.
By \cref{tangency polynomial}, the equality above holds for at most one value of $R'$ (so at most one value of $R$) unless $(x',y')=(x'',y'')$, in which case it clearly holds for every value of $R'$ (so every value of $R$) whenever $r'=r''$.
The latter case means that $C''=\mu_{R,T}(C)$, as requested in case~3a.
Also by \cref{tangency conditions}, the circles $C'$ and $C''$ are internally tangent if and only if
\[\bigl(x''-x'+(R-r^*_T)\cos\phi_T\bigr)^2+\bigl(y''-y'+(R-r^*_T)\sin\phi_T\bigr)^2=(R-r'-\lambda_Tr)^2,\]
which is equivalent to
\[\bigl((x''-x')+R'\cos\phi_T\bigr)^2+\bigl((y''-y')+R'\sin\phi_T\bigr)^2=\bigl((2r^*_T-r'-r'')+R'\bigr)^2,\]
where $R'=R-r^*_T$.
By \cref{tangency polynomial}, the equality above holds for at most one value of $R'$ (so at most one value of $R$) unless $(x',y')=(x'',y'')$.
In the latter case, by condition~1a of \cref{claim:Gallai}, we have $r'=r''$, which implies that $C'$ and $C''$ are externally tangent (as we have shown previously), so they cannot be internally tangent.
\end{claimproof}

Let $R>R_0$ be sufficiently large for the conclusions of \cref{claim:edges} to hold.
Conditions 2 and~3 of \cref{claim:edges} imply the following structure of the tangency graph $G(\calC''_R)$: for every $T\in\calT$, the subgraph induced on the vertices in $\calC'_{R,T}$ is isomorphic to $G(\calC_k)$ and the remaining edges form a collection of matchings between the vertices in $\calC'_{R,T}$ (which are of the form $\nu_{R,T}(C)$ for $C\in\calC_k$) and the vertices in $\calC'_R$ of the form $\mu_{R,T}(C)$ for $C\in\calC_k$.
We exploit this structure in the proofs of the final two claims, which are analogous to \cref{claim:girth-theta,claim:chromatic number-theta}.

\begin{subclaim}\label{claim:girth}
The tangency graph $G(\calC''_R)$ has girth at least $g$.
\end{subclaim}

\begin{claimproof}
Let $G=G(\calC''_R)$.
For every $T\in\calT$, since the subgraph of $G$ induced on the vertices in $\calC'_{R,T}$ is isomorphic to $G(\calC)$, the girth of which is at least $g$ by the induction hypothesis, every cycle in $G$ that lies entirely within $\calC'_{R,T}$ has length at least $g$.
Consider now a cycle in $G$ of length $\ell\geq 3$ that does not lie entirely within $\calC'_{R,T}$ for any $T\in\calT$.
It must contain vertices from $\calC'_R$, say, $C_1,\ldots,C_m$ in this order along the cycle.
For each $i\in[m]$, since $C_i$ has no edges to the rest of $\calC'_R$ and at most one edge to $\calC'_{R,T}$ for each $T\in\calT$, the neighbors of $C_i$ on the cycle lie in two different sets of the form $\calC'_{R,T}$.
For each $i\in[m]$, let $T_i\in\calT$ be such that the part of the cycle between $C_i$ and $C_{i+1}$ (or $C_1$ if $i=m$) lies within $\calC'_{R,T_i}$.
It follows that $(T_1,\ldots,T_m)$ is a cycle in $\calT$ of length $m$ or contains such a cycle if some members of $\calT$ repeat among $T_1,\ldots,T_m$.
Condition~2 of \cref{claim:Gallai} yields $m\geq\lceil\nicefrac{g}{3}\rceil$.
Since there are at least two vertices from $\calC'_{R,T_i}$ between $C_i$ and $C_{i+1}$ (or $C_1$ when $i=m$) for any $i\in[m]$, we conclude that $\ell\geq 3m\geq g$.
\end{claimproof}

\begin{subclaim}\label{claim:chromatic number}
The tangency graph $G(\calC''_R)$ has chromatic number at least $k+1$.
\end{subclaim}

\begin{claimproof}
Suppose for the sake of contradiction that the graph $G=G(\calC''_R)$ is $k$-colorable.
Pick a proper $k$-coloring of $G$, and consider its restriction to the vertices in $\calC'_R$.
It induces a $k$-coloring of $X$ via the correspondence $X\ni(x',y',r')\leftrightarrow\Circle(x',y',R-r')\in\calC'_R$.
By condition~3 of \cref{claim:Gallai}, there is a monochromatic homothetic copy $T$ of $T_0$ in $\calT$, which means that the set of circles $\{\mu_{R,T}(C):C\in\calC_k\}$ is monochromatic.
Since these circles are connected to~$\calC'_{R,T}$ by a perfect matching in $G$, their common color does not occur on the circles in~$\calC'_{R,T}$.
Therefore, the given $k$-coloring of $G$ induces a proper $(k-1)$-coloring of the graph~$G(\calC'_{R,T})$, which is isomorphic to $G(\calC_k)$.
This contradicts the assumption that the graph~$G(\calC_k)$ has chromatic number at least $k$.
\end{claimproof}

We complete the proof of the induction step by setting $\calC_{k+1}=\calC''_R$ and observing that the induction statement follows from Claims \ref{claim:edges} (conditions 1 and~2), \ref{claim:girth}, and~\ref{claim:chromatic number}.
\end{proof}

\setcounter{theorem}{\value{savedtheorem}}

\section{Gallai's theorem with constraints revisited}
\label{sec:Gallai revisited}

This section is devoted to the following strengthening of \cref{Gallai v2}, in which we show that the assertion of the theorem holds for the set of all homothetic copies of $T$ in $X$ rather than for some chosen subset $\calT$ of it.
As was previously mentioned, this strengthening upgrades \cref{Gallai v2} into a generalization of \cref{Gallai}.

\begin{theorem}\label{Gallai v3}
Let\/ $T$ be a finite subset of\/ $\setR^d$ of size at least\/ $3$, let\/ $\calF$ be a countable family of\/ $2d$-variate real polynomials that respects\/ $T$, and let\/ $g$ and\/ $k$ be positive integers.
Then there exists a finite set\/ $X\subset\setR^d$ satisfying the following conditions:
\begin{enumerate}
\item $\calF$ respects\/ $X$,
\item no tuple of fewer than\/ $g$ homothetic copies of\/ $T$ in\/ $X$ form a cycle,
\item every\/ $k$-coloring of\/ $X$ contains a monochromatic homothetic copy of\/ $T$.
\end{enumerate}
\end{theorem}

To provide the proof of \cref{Gallai v3}, we need some preparation.
For two $m$-tuples $(p_1,\ldots,p_m)$ and $(t_1,\ldots,t_m)$ of points in $\setR^d$, let $(p_1,\ldots,p_m)\propto(t_1,\ldots,t_m)$ denote that there is a map $h\colon\setR^d\to\setR^d$ of the form $h(t)=p^*+\lambda t$ for some $p^*\in\setR^d$ and $\lambda\in\setR$ such that $p_i=h(t_i)$ for every $i\in[m]$.
In contrast to the definition of a homothetic map, here we allow $\lambda$ to be negative or zero.
In particular, the constant mapping $h\equiv p^*$ with $\lambda=0$ witnesses that $(p^*,\ldots,p^*)\propto(t_1,\ldots,t_m)$.

\begin{lemma}\label{homothety polynomial}
For any\/ $t_1,\ldots,t_m\in\setR^d$, there is an\/ $md$-variate polynomial\/ $f$ such that for all\/ $p_1,\ldots,p_m\in\setR^d$, we have\/ $f(p_1,\ldots,p_m)=0$ if and only if\/ $(p_1,\ldots,p_m)\propto(t_1,\ldots,t_m)$.
\end{lemma}

\begin{proof}
The condition that $(p_1,\ldots,p_m)\propto(t_1,\ldots,t_m)$ is equivalent to the existence of $\lambda\in\setR$ such that the following conditions are satisfied for all $i\in[m]\setminus\{1\}$ and all $j\in[d]$:
\begin{alignat*}{2}
& (p_i)_j=(p_1)_j && \qquad\text{if }(t_i)_j=(t_1)_j,\\
&\frac{(p_i)_j-(p_1)_j}{(t_i)_j-(t_1)_j}=\lambda && \qquad\text{if }(t_i)_j\neq (t_1)_j,\tag{$*$}
\end{alignat*}
where $(p_i)_j$ denotes the $j$th coordinate of $p_i$.
The fact that the conditions in ($*$) hold with a common value of $\lambda$ simply means that the left sides of these conditions are equal to each other.
With this observation, the conditions above are equivalent to a conjunction of equalities of the form $f_k(p_1,\ldots,p_m)=g_k(p_1,\ldots,p_m)$ for some $md$-variate polynomials $f_1,g_1,\ldots,f_n,g_n$ and some $n\in\setN$.
(Note that the values of $(t_i)_j$ are fixed constants.)
Consider the $md$-variate polynomial $f=\sum_{k=1}^n(f_k-g_k)^2$.
It follows that the conditions above are equivalent to the condition that $f(p_1,\ldots,p_m)=0$, as requested.
\end{proof}

For $m,n\in\setN$, a subset $L$ of the $n$-dimensional $m$-cube $[m]^n$ is called a \emph{generalized combinatorial line} if there exist a non-empty set of indices $J\subseteq[n]$, a choice of $x^*_j\in[m]$ for every $j\in[n]\setminus J$, and a choice of a permutation $\pi_j\colon[m]\to[m]$ for every $j\in J$ such that $L=\{x_i:i\in[m]\}$, where $(x_i)_j=\pi_j(i)$ for $j\in J$ and $(x_i)_j=x^*_j$ for $j\notin J$.
Observe that a combinatorial line (as defined in \cref{sec:gallai}) is a generalized combinatorial line for which all permutations $\pi_i$ are identities.
For emphasis, we will call such a combinatorial line an \emph{ordinary combinatorial line}.
We use the following generalization of \cref{Hales}.

\begin{theorem}[Prömel, Voigt~\cite{PV88}]\label{Hales v2}
For any\/ $m,g,k\in\setN$ with\/ $m\geq 3$, there exist\/ $n\in\setN$ and a set\/ $H\subseteq[m]^n$ such that every\/ $k$-coloring of\/ $H$ contains a monochromatic ordinary combinatorial line of\/ $[m]^n$ and no tuple of fewer than\/ $g$ generalized combinatorial lines of\/ $[m]^n$ contained in\/ $H$ form a cycle.
\end{theorem}

We now have all ingredients to prove \cref{Gallai v3}.

\begin{proof}[Proof of \cref{Gallai v3}]
Put $T=\{t_1,\ldots,t_m\}$, where $m=|T|$.
Let $n$ and $H\subseteq[m]^n$ be as claimed in \cref{Hales v2} applied to $m$, $g$, and $k$.
For a given vector $\gamma=(\gamma_1,\ldots,\gamma_n)\in \setR^n$, define a map $\zeta_\gamma\colon H\to\setR^d$ by $\zeta_\gamma(x)=\sum_{i=1}^n\gamma_it_{x_i}$, and put $X_\gamma=\zeta_\gamma(H)=\{\zeta_\gamma(x):x\in H\}\subset\setR^d$.

Let $f$ be the $md$-variate polynomial claimed in \cref{homothety polynomial} for $t_1,\ldots,t_m$.
For any fixed $x_1,\ldots,x_m\in H$ and all $i\in[m]$, consider the polynomial $p_{x_i}=\zeta_\gamma(x_i)=\sum_{j=1}^n\gamma_jt_{(x_i)_j}$ in variables $\gamma_1,\ldots,\gamma_n$.
We claim that the polynomial $f_{x_1,\ldots,x_m}$ in variables $\gamma_1,\ldots,\gamma_n$ given by
\[\textstyle f_{x_1,\ldots,x_m}=f(p_{x_1},\ldots,p_{x_m})=f\bigl(\sum_{j=1}^n\gamma_jt_{(x_1)_j},\;\ldots,\;\sum_{j=1}^n\gamma_jt_{(x_m)_j}\bigr)\]
is not identically zero unless $\{x_1,\ldots,x_m\}$ is a generalized combinatorial line of $[m]^n$.

Indeed, suppose $f_{x_1,\ldots,x_m}(\gamma)=0$ for all $\gamma\in\setR^n$.
Consider an index $j\in[m]$, and let $\gamma_j=1$ and $\gamma_{j'}=0$ for every $j'\neq j$.
It follows that $f(t_{(x_1)_j},\ldots,t_{(x_m)_j})=0$, so $(t_{(x_1)_j},\ldots,t_{(x_m)_j})\propto(t_1,\ldots,t_m)$.
This is possible only in one of the following three cases, corresponding to scaling factors $\lambda=1$, $\lambda=0$, and $\lambda=-1$ in the definition of $\propto$, respectively:
\begin{alphaenumerate}
\item $((x_1)_j,\ldots,(x_m)_j)=(1,\ldots,m)$,
\item $(x_1)_j=\cdots=(x_m)_j$,
\item the set $T$ is centrally symmetric, and $((x_1)_j,\ldots,(x_m)_j)=(\pi(1),\ldots,\pi(m))$ where $\pi$ is a permutation of $T$ that corresponds to the central symmetry of $T$.
\end{alphaenumerate}
The fact that one of these conditions holds for every $j\in[m]$ implies that $\{x_1,\ldots,x_m\}$ is a generalized combinatorial line of $[m]^n$, as claimed.

We continue as in the proof of \cref{Gallai v2} with the following additional detail: to the family of polynomials on which \cref{zero set} is applied, we add the polynomials $f_{x_1,\ldots,x_m}$ for all choices of $x_1,\ldots,x_m\in H$ which \emph{do not} form a generalized combinatorial line of $[m]^n$.
By the claim above, these polynomials are not identically zero, so \cref{zero set} can be applied.
As a result, by the argument from the proof of \cref{Gallai v2}, we obtain a vector $\gamma\in\setR^n$ which yields a map $\zeta=\zeta_\gamma$ that is injective and a set $X=X_\gamma=\zeta(H)$ such that conditions 1 and~3 of \cref{Gallai v3} hold for $X$.
Moreover, $f_{x_1,\ldots,x_m}(\gamma)\neq 0$ for all choices of $x_1,\ldots,x_m\in H$ which do not form a generalized combinatorial line of $[m]^n$.

It remains to show that condition~2 holds for $X$.
Every homothetic copy of $T$ in $X$ is a set $\{p_1,\ldots,p_m\}\subseteq X$ such that $(p_1,\ldots,p_m)\propto(t_1,\ldots,t_m)$, which implies that $f(p_1,\ldots,p_m)=0$.
Let $x_1,\ldots,x_m\in H$ be such that $p_i=\zeta(x_i)=p_{x_i}(\gamma)$.
Since $f(p_1,\ldots,p_m)=0$, the set $\{x_1,\ldots,x_m\}$ is a generalized combinatorial line of $[m]^n$.
Since no tuple of fewer than $g$ generalized combinatorial lines contained in $H$ form a cycle and $\zeta$ is injective, it follows that no tuple of fewer than $g$ homothetic copies of $T$ in $X$ form a cycle, which is condition~2.
\end{proof}

\section*{Acknowledgments}

This work was initiated at the online workshop ``Geometric graphs and hypergraphs''.
We thank the organizers Torsten Ueckerdt and Yelena Yuditsky for a very nice workshop and all participants for fun coffee breaks and a fruitful atmosphere.
We thank the anonymous reviewers for their helpful suggestions.

\bibliography{Ringel}

\begin{thebibliography}{10}

\bibitem{AH77}
Kenneth Appel and Wolfgang Haken.
\newblock Every planar map is four colorable. {P}art {I}: {D}ischarging.
\newblock {\em Illinois Journal of Mathematics}, 21(3):429--490, 1977.

\bibitem{AHK77}
Kenneth Appel, Wolfgang Haken, and John Koch.
\newblock Every planar map is four colorable. {P}art {II}: {R}educibility.
\newblock {\em Illinois Journal of Mathematics}, 21(3):491--567, 1977.

\bibitem{Coxeter61}
Harold Scott~MacDonald Coxeter.
\newblock {\em Introduction to Geometry}.
\newblock John Wiley \& Sons, 1961.

\bibitem{Davies21}
James Davies.
\newblock Box and segment intersection graphs with large girth and chromatic
  number.
\newblock {\em Advances in Combinatorics}, 2021:7, 9 pp., 2021.
\newblock \href {https://doi.org/10.19086/aic.25431}
  {\path{doi:10.19086/aic.25431}}.

\bibitem{AdG18}
Aubrey D.~N.~J. de~Grey.
\newblock The chromatic number of the plane is at least 5.
\newblock {\em Geombinatorics}, 28:5--18, 2018.

\bibitem{Descartes47}
Blanche Descartes.
\newblock A three colour problem.
\newblock {\em Eureka}, 9(21):24--25, 1947.

\bibitem{Descartes54}
Blanche Descartes.
\newblock Solution to advanced problem no. 4526.
\newblock {\em The American Mathematical Monthly}, 61:352, 1954.

\bibitem{FK91}
Hillel Furstenberg and Yitzhak Katznelson.
\newblock A density version of the {H}ales-{J}ewett theorem.
\newblock {\em Journal d’Analyse Math\'ematique}, 57:64--119, 1991.
\newblock \href {https://doi.org/10.1007/BF03041066}
  {\path{doi:10.1007/BF03041066}}.

\bibitem{HJ63}
Andrew~W. Hales and Robert~I. Jewett.
\newblock Regularity and positional games.
\newblock {\em Transactions of the American Mathematical Society},
  106(2):222--229, 1963.

\bibitem{JR84}
Brad Jackson and Gerhard Ringel.
\newblock Colorings of circles.
\newblock {\em The American Mathematical Monthly}, 91(1):42--49, 1984.
\newblock URL: \url{https://www.jstor.org/stable/2322168}, \href
  {https://doi.org/10.1080/00029890.1984.11971333}
  {\path{doi:10.1080/00029890.1984.11971333}}.

\bibitem{JT95}
Tommy~R. Jensen and Bjarne Toft.
\newblock {\em Graph Coloring Problems}.
\newblock John Wiley \& Sons, 1995.

\bibitem{Kalai15}
Gil Kalai.
\newblock Some old and new problems in combinatorial geometry {I}: {A}round
  {B}orsuk's problem.
\newblock In Artur Czumaj, Agelos Georgakopoulos, Daniel Kr\'a{\v{l}}, Vadim
  Lozin, and Oleg Pikhurko, editors, {\em Surveys in Combinatorics}, volume 424
  of {\em London Mathematical Society Lecture Note Series}, pages 147--174.
  Cambridge University Press, 2015.

\bibitem{Koebe36}
Paul Koebe.
\newblock Kontaktprobleme der konformen {A}bbildung.
\newblock {\em Berichte \"uber die Verhandlungen der S\"achsischen Akademie der
  Wissenschaften zu Leipzig, Mathematisch-Physische Klasse}, 88:141--164, 1936.

\bibitem{KN99}
Alexandr~V. Kostochka and Jaroslav Ne{\v{s}}et{\v{r}}il.
\newblock Properties of {D}escartes' construction of triangle-free graphs with
  high chromatic number.
\newblock {\em Combinatorics, Probability and Computing}, 8(5):467--472, 1999.
\newblock \href {https://doi.org/10.1017/S0963548399004022}
  {\path{doi:10.1017/S0963548399004022}}.

\bibitem{Mityagin20}
Boris~S. Mityagin.
\newblock The zero set of a real analytic function.
\newblock {\em Matematicheskie Zametki}, 107(3):473--475, 2020.

\bibitem{Pach17}
J\'anos Pach.
\newblock Finite point configurations.
\newblock In Jacob~E. Goodman, Joseph O'Rourke, and Csaba~D. T\'oth, editors,
  {\em Handbook of Discrete and Computational Geometry}, pages 26--50. CRC
  Press, 3rd edition, 2017.

\bibitem{PV88}
Hans~J\"urgen Pr\"omel and Bernd Voigt.
\newblock A sparse {G}raham-{R}othschild theorem.
\newblock {\em Transactions of the American Mathematical Society},
  309(1):113--137, 1988.
\newblock \href {https://doi.org/10.1090/S0002-9947-1988-0957064-5}
  {\path{doi:10.1090/S0002-9947-1988-0957064-5}}.

\bibitem{PV90}
Hans~J\"urgen Pr\"omel and Bernd Voigt.
\newblock A sparse {G}allai-{W}itt theorem.
\newblock In Rainer Bodendiek and Rudolf Henn, editors, {\em Topics in
  Combinatorics and Graph Theory}, pages 747--755. Physica-Verlag Heidelberg,
  1990.
\newblock \href {https://doi.org/10.1007/978-3-642-46908-4_84}
  {\path{doi:10.1007/978-3-642-46908-4_84}}.

\bibitem{Rado45}
Richard Rado.
\newblock Note on combinatorial analysis.
\newblock {\em Proceedings of the London Mathematical Society}, 48:122--160,
  1945.

\bibitem{Ringel59}
Gerhard Ringel.
\newblock {\em F\"{a}rbungsprobleme auf Fl\"{a}chen und Graphen}, volume~2 of
  {\em Mathematische Monographien}.
\newblock VEB Deutscher Verlag der Wissenschaften, 1959.

\bibitem{Sachs94}
Horst Sachs.
\newblock Coin graphs, polyhedra, and conformal mapping.
\newblock {\em Discrete Mathematics}, 134(1--3):133--138, 1994.
\newblock \href {https://doi.org/10.1016/0012-365X(93)E0068-F}
  {\path{doi:10.1016/0012-365X(93)E0068-F}}.

\bibitem{Soifer09}
Alexander Soifer.
\newblock {\em The Mathematical Coloring Book: Mathematics of Coloring and the
  Colorful Life of its Creators}.
\newblock Springer, 2009.

\bibitem{VanderWaerden27}
Bartel~L. van~der Waerden.
\newblock Beweis einer {B}audetschen {V}ermutung.
\newblock {\em Nieuw Archief voor Wiskunde}, 15:212--216, 1927.

\end{thebibliography}

\end{document}